\documentclass{article}
\usepackage{amsmath,amsthm,amssymb,latexsym}
\usepackage[active]{srcltx}

\title{{\bf Lipschitz and biLipschitz Maps on Carnot Groups}}
\author{\Large {William Meyerson}}
\date{}
\newtheorem{teo}{Theorem}[section]
\newtheorem{co}[teo]{Corollary}

\newtheorem{defi}[teo]{Definition}
\newtheorem{note}[teo]{Note}

\newtheorem{prop}[teo]{Proposition}
\newtheorem{fac}[teo]{Fact}

\begin{document}
\maketitle
\begin{abstract}
Suppose $A$ is an open subset of a Carnot group $G$ and $H$ is another 
Carnot group.  We show that a Lipschitz function from $A$ to $H$ whose 
image has positive Hausdorff measure in the appropriate dimension is 
biLipschitz on a subset of $A$ of positive Hausdorff measure.  
We also construct Lipschitz maps from open sets in Carnot groups to 
Euclidean space that do not decrease dimension.  Finally, we discuss 
two counterexamples to explain why Carnot group structure is 
necessary for these results.

Primary subject:  43A80

\end{abstract}

\section{Introduction}\label{sec1}
In 1988, Guy David proved in \cite{D1} that if $f$ is a Lipschitz function 
from the unit cube in $\textbf{R}^n$ to a subset of some Euclidean space 
with positive $n$-dimensional Hausdorff measure, there exists a subset $K$ of 
the domain of $f$ with positive $n$-dimensional Hausdorff measure such 
that $f$ is biLipschitz on $K$.

Shortly thereafter, Peter Jones proved the following stronger result in 
\cite{J}:  if $f$ is a Lipschitz function from the unit cube in 
$\textbf{R}^n$ to a 
subset of some Euclidean space, then the unit cube can be broken up into 
the union of a `garbage' set (whose image under $f$ has arbitarily small 
$n$-dimensional Hausdorff content) and a finite number of sets $K_1, 
\dots, K_N$ such that $f$ is biLipschitz on each $K_i$.

  Three years later, Guy David in \cite{D2} translated this proof 
into the language of wavelets, which are more readily generalizable 
to Heisenberg and other Carnot groups.  The proof as 
written in \cite{D2} only depends on a few general properties, all 
but one of which hold for Heisenberg (and other Carnot) groups.  

This story has further generalizations: for example, \cite{DS} generalizes Jones' argument to work with Lipschitz functions that are only defined on Ahlfors $d$-regular subsets of a Euclidean space $\textbf{R}^N$, with $d$ possibly less than $N$, while \cite{Se} allows the domain and range to be metric spaces subject to a specific condition. 

The second section of this paper adapts some of the ideas in 
\cite{D2} and \cite{J} to Carnot groups and then proves the ``Lipschitz implies biLipschitz'' result mentioned first in the abstract.

The third section will investigate the 
question of how big, in terms of dimension, Lipschitz images of Carnot 
groups in Euclidean space can be.  

Finally, the fourth section will explore two counterexamples explaining why Carnot group structure is 
necessary for these results.  In particular, neither Ahlfors regularity nor sub-Riemannian manifold structure would be sufficient.

The author would also like to thank Raanan Schul for showing him Guy 
David's proof of the decomposition of Lipschitz functions from subsets of
$\textbf{R}^m$ to $\textbf{R}^n$ into biLipschitz pieces.  Further, the 
author would like to thank Jeremy Tyson for showing him the usefulness of 
Cantor set constructions for creating maps on Carnot groups.

\section{Jones-Type Decomposition \\ for Carnot Groups}\label{sec2}

\subsection{Brief Outline}\label{sec21}

This section is organized as follows.  In Section 2.2 we shall give some 
definitions concerning Carnot groups and set up some notational conventions.  In Section 2.3 we shall 
state the five properties of Euclidean space on which David's argument 
rests and 
show how the first four of them work for Heisenberg groups.  In Section 
2.4 we will explain why these properties also work for other Carnot 
groups.  In Section 2.5, we shall prove our main result:  if $A$ is
an appropriate subset of the $k$th Heisenberg group $H_k$ corresponding 
roughly to the unit cube in $\textbf{R}^n$, and $F$ is a Lipschitz 
function from $A$ to another Heisenberg group whose image has positive 
Hausdorff $(2k + 2)$-dimensional measure, then there exists $B \subset A$ 
with positive Hausdorff $(2k + 2)$-dimensional measure such that $F$ is 
biLipschitz on $B$.  Finally, in Section 2.6 we will derive some corollaries of the main theorem from Section 2.5.  

Although our main focus is on the Heisenberg groups (especially $H_1$), 
all of the results in this paper apply equally well to Carnot groups in 
general.  To exploit this fact, the results in Section 2.5 will be stated 
and proved in the more general context of Carnot groups.

\subsection{Definitions}\label{sec22}

We begin by defining the Heisenberg groups.

\begin{defi}
The \textbf{nth Heisenberg group} $H_n$ is defined as the set 

$$\{(z_1, \dots, z_n, t): z_j \in \textbf{C}, t \in \textbf{R}\}$$

\noindent equipped with the following group law:

$$(z_1, \dots, z_n, t)(w_1, \dots, w_n, s) = (z_1 + w_1, \dots, z_n + w_n, 
t + s + \Im{\Sigma_{j = 1}^n z_j \bar{w_j}})$$

\noindent where $\Im$ denotes imaginary part.
\end{defi}

For $n = 1$, we often write $z_1$ in terms of its real components as $z_1 = x + iy$ and refer to the point $(z_1, t)$ as $(x, y, t)$, so $H_1$ inherits a natural Euclidean coordinate structure from $\textbf{R}^3$.

The Heisenberg group is a special example of a Carnot group, which is 
defined as follows:

\begin{defi}
A \textbf{Carnot group} $G$ is a connected, simply connected, nilpotent 
Lie group whose Lie algebra $\mathfrak{g}$ is graded, i.e.

$$\mathfrak{g} = \oplus_{j = 1}^d \mathfrak{g}_j$$ where
$$[\mathfrak{g}_1, \mathfrak{g}_j] = \mathfrak{g}_{j + 1}$$ and
$$\mathfrak{g}_{d+1} = \{0\}.$$ 
 We call $\mathfrak{g}_1$ 
the \textbf{horizontal} component of $\mathfrak{g}$.
\end{defi}

By standard results of Lie group theory (see, for example, \cite{V}), 
the exponential map gives a diffeomorphism between a Carnot group and its 
Lie algebra.  Further, the standard definition of a Lie algebra 
in terms of vector fields provides a canonical identification between the 
tangent space of a Lie group at a given point and the Lie group itself.  
(When $g \in G$ is fixed, for every tangent vector $v$ there is a unique 
$X \in \mathfrak{g}$ such that $X(g) = v$ and we can identify 
$\textrm{exp}(X)$ with $v$.)

We shall freely use these canonical identifications betwen a Carnot group, its Lie algebra, and its tangent space throughout this paper.  For 
example, every Carnot group has a coordinate structure induced by its Lie algebra.  For $H_n$, this coordinate structure was already mentioned in Definition 2.1, where $\mathfrak{g}_1$ consists of the 
points of the form $(z_1, \dots, z_n, 0)$ with final coordinate equal to 
zero.

Every Carnot group has a family of dilation homomorphisms $\{\delta_\lambda: \lambda > 0\}$ and a metric called the Carnot-Carath\'eodory metric.  They are defined as follows:

\begin{defi}
Let $\lambda > 0$, let $G$ be a Carnot group and let $g \in G$, where
$$g = \Sigma_i g_i$$
with $g_i \in \mathfrak{g}_i$.  Define the \textbf{dilation} 
$$\delta_\lambda(g) = \Sigma_i \lambda^i g_i.$$ 

\end{defi}

\begin{defi}
Let $G$ be a Carnot group, let $g, h \in G$, and let $\Gamma_{g,h}$ to be the set of all curves
$$\gamma:  [0, 1] \rightarrow G$$
with $\gamma(0) = g$, $\gamma(1) = h$, and $\gamma'(t) \in \mathfrak{g}_1$ 
for each $t \in [0, 1]$.
Define the \textbf{Carnot-Carath\'eodory} distance between $g$ and $h$ to be 
$$d_{CC}(g, h) = \inf_{\gamma \in \Gamma_{g, h}} \int_0^1 
|\gamma'(t)| dt$$
where $|\gamma'(t)|$ is the length of $\gamma'(t)$ in a fixed Euclidean metric on the real 
vector space $\mathfrak{g}_1$.
\end{defi} 
\noindent Because $\Gamma_{g, h}$ in the above definition is nonempty (cf 
\cite{Mo}), $d_{CC}(g, h) < \infty$ whenever $g, h \in G$.

\begin{note}
It is often easier to work with a comparable $L^{\infty}$ 
quasidistance function $d$ based on the Carnot metric.  For the first 
Heisenberg group $H_1$, this is done by defining distance to 
the origin as

$$d((x, y, z), (0, 0, 0)) = \textrm{max}(|x|, |y|, |z|^{.5})$$
and for an arbitrary $g, h$ in this group, defining

$$d(g, h) = d(h^{-1} g, (0, 0, 0)).$$
\end{note}

There is of course a completely analogous construction in an arbitrary 
Carnot group:  if $G$ is a Carnot group, we use the grading of its Lie 
algebra $\mathfrak{g}$ as in the definition of Carnot groups:

$$\mathfrak{g} = \oplus_{j = 1}^d \mathfrak{g}_j.$$
Because the identity element in a Carnot group is the image of the 
origin under the exponential map, we shall refer to it as $0$.  
Now, letting $g$ be an arbitrary point in $G$ we first define its 
quasidistance to the identity element, $d(g, 0)$, by recalling the 
direct sum decomposition

$$\exp^{-1}(g) = \Sigma_j g_j$$

\noindent with $g_j \in \mathfrak{g_j}$ and setting 
$$d(g, 0) = \textrm{max}_{1 \leq j \leq d} (||g_j||_j)^{\frac{1}{j}}.$$

\noindent where $||\cdot||_j$ is a norm on $\mathfrak{g}_j$ 
for $j = 1, \dots, d$.

\noindent Finally, for an arbitrary $g, h \in G$, we finish by setting 

$$d(g, h) = d(h^{-1} g, 0).$$

\noindent For the duration of this paper, $d_{CC}$ shall refer to 
Carnot-Carath\'eodory distance and $d$ shall refer to quasidistance.

A fundamental operation for Carnot groups is the Pansu differential, defined as follows (see \cite{Ca} for example):

\begin{defi}

Let $F:  G \rightarrow H$ be a function from one Carnot group $G$ to 
another Carnot group $H$.  The \textbf{Pansu differential} $DF(g)$  
of $F$ at $g \in G$ is the map

$$DF(g): G \rightarrow H$$ defined at $g' \in G$ as the limit

$$DF(g)(g') = \lim_{s \rightarrow 0} \delta_{s^{-1}} [F(g)^{-1} F(g 
\delta_s g')]$$
whenever the limit exists.

\end{defi}

Using the canonical identifications stated above, we can view the Pansu 
differential as a map between Lie algebras or as a map from the tangent 
space at $g \in G$ to the tangent space at $F(g)$.  We shall take 
advantage of this fact throughout the paper.

Further, in the tangent vector interpretation, the Pansu differential 
$DF(g)$ induces a linear map between the horizontal component of the 
tangent space of $G$ at $g$ and the horizontal component of the tangent 
space of $H$ at $F(g)$ (see \cite{P}).  Calling this linear map $MF(g)$, 
we can view $MF$ as a matrix-valued map sending $g$ to $MF(g)$.

 \subsection{Five key properties}\label{sec23}

\subsubsection{Dyadic decomposition}\label{sec231}

There exists a dyadic decomposition for Euclidean space defined as follows:  For each nonnegative integer $k$ we let $\mathcal{Q}_k$ be 
the set of all ``cubes" of the form 
$$(a_1 \cdot 2^{-k}, (a_1 + 1) \cdot 2^{-k}) \times \dots \times (a_n 
\cdot 2^{-k}, (a_n + 1) \cdot 2^{-k})$$ 
contained in the unit cube, where the $a_i$ are all integers.  Then the 
elements of $\mathcal{Q}_k$ are disjoint open sets.  
Further, each element of $\mathcal{Q}_k$ is (up to a set of measure zero) a 
disjoint union of elements of $\mathcal{Q}_{k + 1}$, the $\mathcal{Q}_k$'s are all translates 
of each other, and one can transform an arbitrary element of $\mathcal{Q}_k$ into 
an arbitrary element of $\mathcal{Q}_{k + 1}$ 
by a dilation (by a factor of $2^{-1}$) followed by translation.  
Finally, fixing a cube $Q \in \mathcal{Q}_k$ and letting $d$ be its diameter (i.e. 
$d = \sqrt n 2^{-k}$), the 
number of cubes in $\mathcal{Q}_k$ whose distance from $Q$ is at most $d$
is bounded above by a constant depending only on $n$.  

Our immediate goal is to generalize this decomposition to the Heisenberg 
group $H_1$.
To do this we loosely follow Christ's construction of Theorem 11 in 
\cite{Ch}.  First we let $B_0$ denote the discrete subgroup of $H_1$ generated by $(1, 0, 0)$ and $(0, 1, 0)$ and call it the \textbf{discrete Heisenberg group}.  We then define $B_n$, for each positive integer $n$, to be the image of $B_0$ under 
the dilation $\delta_{10^{-n}}$ (in particular, the first $2$ 
coordinates are multiplied by $10^{-n}$; the final coordinate is 
multiplied by 
$10^{-2n}$).  Equivalently, $B_n$ is the subgroup of the first
Heisenberg group generated by $(10^{-n}, 0, 0)$ and $(0, 10^{-n}, 0)$.  \textbf{If 
$x$ is a point in $B_n$, we give it 
the label $(x, n)$} and note that $x$ has a different label for each 
$B_n$ containing $x$.  We form 
a tree by defining an order relation $\leq$ on the set of all such pairs 
$(x, n)$.  We start this procedure with the following 
definition.

\begin{defi}  
$(x, \alpha)$ is a \textbf{parent} of $(y, \beta)$ if $\beta = \alpha + 1$ 
and $y = xg$ where the first two components of $g$ all lie in 
$(-\frac{1}{2} 10^{-\alpha}, \frac{1}{2} 10^{-\alpha}]$ and the final 
component lies in $(-\frac{1}{2} \cdot 10^{-2\alpha}, \frac{1}{2} \cdot 
10^{-2\alpha}]$.  
\end{defi}

Using the obvious analogies from family trees (`ancestor', `descendant', 
`grandparent', `sibling', etc.) for 
both the tree points and corresponding 
dyadic cubes (to be defined momentarily), we say $(x, \alpha) \leq (y, \beta)$ if $(y, \beta)$ is an 
ancestor of $(x, \alpha)$.  Following along exactly as in Definition 14 
of \cite{Ch}, we create 
from this tree a family of dyadic `cubes'.  In particular, we define
$$Q(x, \alpha) = \bigcup_{(y, \beta) \leq (x, \alpha)} B_{CC}(y, \frac{1}{10} 10^{-\beta})$$
where $B_{CC}(z, \epsilon)$ is the ball centered at $z$ of radius $\epsilon$ with respect to Carnot-Carath\'eodory distance.  We will say that each cube $Q(x, 
\alpha)$ is a \textbf{cube at scale $\alpha$} and we define $\mathcal{Q}_\alpha$ to be the set of all the cubes of scale $\alpha$.  All the cubes in $\mathcal{Q}_\alpha$ are translates of each 
other by elements of the discrete Heisenberg group of the appropriate 
scale; further, each member of 
each $\mathcal{Q}_\alpha$ is an open set while each element of $\mathcal{Q}_\alpha$ is (up to 
a set of measure zero) the disjoint union of elements of $\mathcal{Q}_{\alpha + 1}$.  Also, 
one can transform an arbitrary element of $\mathcal{Q}_\alpha$ into an arbitrary 
element of $\mathcal{Q}_{\alpha + 1}$ by 
a dilation (by a factor of $10^{-1}$) followed by translation.  Finally, 
the number of cubes in $\mathcal{Q}_\alpha$ within diam$(Q(x, \alpha))$ of $Q(x, 
\alpha)$ is bounded by a constant independent of $\alpha$.

Analogously, for the $k$th Heisenberg group, we begin by rewriting the elements of $H_k$ to 
mirror the above construction for $H_1$:  in other words, writing 
$z_j = x_{2j - 1} + i x_{2j}$ where $x_{2j - 1}, x_{2j} \in 
\textbf{R}$, we let $B_0$ be the subgroup of $H_k$ generated by

$$\{(x_1, \dots, x_{2k}, 0): x_j = \pm \delta_{j, l}, 1 \leq l \leq 2k\}$$ 

\noindent where $\delta_{j,l}$ is the Kronecker delta.  In this setting, $B_n$ would be the subgroup of $H_k$ generated by 
$$\{(x_1, \dots, x_{2k}, 0): x_j = \pm 10^{-n} \delta_{j, l}, 1 \leq l \leq 2k\}$$
and the construction for $H_1$ goes through for $H_k$ with only minor changes.  In particular, the definition of $(x, \alpha)$ being a parent of $(y, \beta)$ would now require $y = xg$ where the first $2k$ components of $g$ all lie in 
$(-\frac{1}{2} 10^{-\alpha}, \frac{1}{2} 10^{-\alpha}]$ and the final component lies in $(-\frac{1}{2} \cdot 10^{-2\alpha}, \frac{1}{2} \cdot 10^{-2\alpha}]$.  

In this construction, the analogue to the unit cube in Euclidean space is 
the unique cube of scale $0$ containing the identity element; according to 
the notation defined in the preceding paragraph, the name for this cube is 
$Q(0, 0)$.
\medskip

\noindent
\textbf{Remark}:  In making this decomposition we are saying 
nothing about the boundaries of the elements of the $\mathcal{Q}_\alpha$ other than 
that they are closed sets of Hausdorff measure zero in the 
appropriate dimension.  Also, this decomposition is not the same as the 
decomposition of the Heisenberg group found in \cite{Str}.

\subsubsection{Orthogonal decomposition of $L^2$}\label{sec232}

Looking back at Euclidean space $\textbf{R}^n$ for inspiration, we note 
that the Hilbert space $L^2([0, 1]^n)$ of square-integrable functions on 
the unit cube can be decomposed into orthogonal subspaces as follows:  if $\beta$ is a positive integer, we define $C_\beta \subset L^2([0, 1]^n)$ as

$$\{f \in L^2([0, 1]^n):  f|_Q \textnormal{ is constant for } Q \in \mathcal{Q}_\beta  \textnormal{ and } \int_Q f = 0 \textnormal{ for } Q \in \mathcal{Q}_{\beta - 1}\}$$

\noindent
while $C_0 \subset L^2([0, 1]^n)$ is defined as 
$$\{f \in L^2([0, 1]^n):  f \textnormal{ is constant}\}.$$
\noindent
This yields the orthogonal decomposition

$$L^2([0, 1]^n) = \oplus_{\beta = 0}^\infty C_\beta.$$
\noindent
In other words, if $f \in 
C_\beta, g \in C_\gamma$ 
with $\beta \neq \gamma$, $\int_{[0, 1]^n} fg = 0$ while for each $h \in 
L^2([0, 1]^n)$ there exists $h_\beta \in C_\beta$ for $\beta$ a 
nonnegative integer with
$$h = \Sigma_{\beta = 0}^\infty h_\beta$$ 
with the sum in question converging in $L^2([0, 1]^n)$ to $h$.

For the Heisenberg groups we can mimic this procedure as follows:  
here, 
our `base' cube shall be denoted as $Q(0, 0)$ where the first zero 
denotes the origin and the second zero denotes scale.  Similarly, we 
define the $C_\beta$ (as subspaces of the Hilbert space $L^2(Q(0, 0))$ of 
real-valued, square-integrable functions) identically to the way we did 
with Euclidean space.  In other words, if $\beta$ is a positive integer, we define $C_\beta \subset L^2(Q(0, 0))$ as

$$\{f \in L^2(Q(0, 0)):  f|_Q \textnormal{ is constant for } Q \in \mathcal{Q}_\beta  \textnormal{ and } \int_Q f = 0 \textnormal{ for } Q \in \mathcal{Q}_{\beta - 1}\}$$

\noindent
while $C_0 \subset L^2(Q(0, 0))$ is defined as 
$$\{f \in L^2(Q(0, 0)):  f \textnormal{ is constant}\}.$$
\noindent
This yields the orthogonal decomposition

$$L^2(Q(0, 0)) = \oplus_{\beta = 0}^\infty C_\beta.$$

For $\beta > 0$, $C_\beta$ has a spanning set consisting of the functions 
$f_{Q, Q'}$ for $Q, Q'$ sibling cubes in $\mathcal{Q}_\beta$ 
defined as follows:  $f_{Q, Q'}$ is equal to $1$ on $Q$, $-1$ on $Q'$, and 
$0$ everywhere else; we shall call this spanning set $S_\beta$.  $S_\beta$ 
is \textbf{approximately orthogonal} in the following sense:
there exists some universal constant $K$ (independent of $\beta$) such 
that for each $f \in S_\beta$ we have

$$\#\{g \in S_\beta:  \int_{Q(0, 0)} fg \neq 0\} \leq K$$
\noindent where the $\#$ symbol denotes cardinality.

When we proceed to the proof, we will wish to find a fixed $Y > 0$ such that if 
$g, g' \in Q(0, 0)$ with $g \neq g'$, there exists some dyadic cube $Q$ such that

$$\textnormal{ diam }(Q) < Y d_{CC}(g, g') \textnormal{ and } g, g' \in Q.$$
This is arranged by considering not just the cube families $\mathcal{Q}_\alpha$ 
discussed in the previous section but expanding each cube family $\mathcal{Q}_\alpha$ to a larger family $\mathcal{Q}'_\alpha$.

If $\alpha > 0$, we define $\mathcal{Q}'_\alpha$ to consist of the the cubes of the form
$$\{gQ:  Q \in \mathcal{Q}_\alpha, g \in B_{\alpha + 2}\};$$ 
remember that $B_{\alpha + 2}$ was defined in the previous subsection as the discrete Heisenberg group of scale $\alpha + 2$.

This does not cause the number of dyadic cubes of a given scale to multiply unreasonably because writing $g \in B_{\alpha + 2}$ in coordinate form as $(z_1, \dots, z_n, t)$, every element of $\mathcal{Q}'_\alpha$ can be written as $gQ$ for some $Q \in \mathcal{Q}_\alpha$ and $g \in B_{\alpha + 2}$ with
$$z_1, \dots, z_n \in [-10^{-\alpha}, 10^{-\alpha}] \textrm{ and } t \in [-10^{-2\alpha}, 10^{-2\alpha}].$$
Letting $L_{g}$ denote left translation by $g$ whenever $g \in H^k$, we then define

$$C'_\beta = \{f \circ L_{g^{-1}}:  f \in C_\beta, g \in B_{\beta + 2}\}.$$
In fact, writing $g \in B_{\beta + 2}$ in coordinate form as $(z_1, \dots, z_n, t)$, every element of $C'_\beta$ can be written as $f \circ L_{g^{-1}}$ for some $f \in C_\beta$ and $g \in B_{\beta + 2}$ with
$$z_1, \dots, z_n \in [-10^{-\beta}, 10^{-\beta}] \textrm{ and } t \in [-10^{-2\beta}, 10^{-2\beta}].$$
Fixing $\beta$, we can construct an approximately 
orthogonal basis for $C'_\beta$ analogously to the way we did 
for each $C_\beta$:  we simply construct an approximately orthogonal basis for $C_\beta \circ L_{g^{-1}}$ for each $g$ separately.

Finally, for both Euclidean space and the Heisenberg group, it is 
occasionally 
necessary to work with sets on a slightly larger scale than the unit 
cube.  To do this, one fixes some integer $k \leq 0$, denotes our base 
cube to be the cube of scale $k$ which contains $Q(0, 0)$, and then 
defines $C_\beta$ and $C'_\beta$ appropriately for $\beta \leq 0$ (for 
example, $C_k$ will consist of the constant functions on our new base 
cube here).

\subsubsection{Differentiability}\label{sec233}
On the Euclidean unit cube, there exists a Jacobian map that sends each 
Lipschitz function $f$ (which may be either scalar-valued or Euclidean 
vector-valued) on the unit cube to the almost-everywhere-defined 
function $Jf$, the Jacobian of $f$.  
At almost every point, the Jacobian is a linear map from 
the tangent space of the domain to the tangent space of the image.  
Further, the partial derivative 
of each component is bounded above by the Lipschitz coefficient of $f$.  
Finally, a Lipschitz function $f$ with almost everywhere constant Jacobian 
defined on a 
connected open set is uniquely determined by this Jacobian and its value 
at a single point:  if $Jf$ is equal to the linear map $T$ almost 
everywhere and $f(x_0) = y_0$ then

$$f(x) = T(x - x_0) + y_0 \textrm{ for all } x \textrm{ where } f(x) 
\textrm{ is defined.}$$

Similarly, if $G$ and $H$ are two Heisenberg groups and  $F:  G 
\rightarrow H$ is Lipschitz, then by \cite{P} the Pansu differential $DF$ (which, for almost every $g \in G$ induces a map $DF(g):  G \rightarrow H$) satisfies these three properties:

\bigskip

\noindent (i) At almost every $g \in G$, the differential of the Lie group map $DF(g)$ at the identity induces a Lie 
algebra homomorphism from the tangent space of $G$ at $g$ to the tangent space of $H$ at $F(g)$.

\bigskip

\noindent (ii) The magnitude of each component of $DF$ is bounded 
above (up to a constant depending on normalization) by the Lipschitz coefficient of 
$F$.   

\bigskip

\noindent (iii) If for almost every $g$ with respect to Haar 
measure on $G$, $DF(g)$ (which was defined as an $H$-valued function 
defined on $G$) is equal to the Lie group homomorphism $\phi:  G 
\rightarrow H$ and $g_0 \in G$, $h_0 \in H$ with $F(g_0) = h_0$ 
then

$$F(g) = h_0 \phi(g_0^{-1} g) \textrm{ for all } g \textrm{ where } F(g) 
\textrm{ is defined.}$$

Of the properties, only (iii) is not a simple consequence of 
\cite{P}.  However, (iii) is a direct consequence of the following fact concerning 
uniqueness of Lipschitz maps:

\begin{fac}
Suppose $G$ and $H$ are Carnot groups, $U \subset G$ is connected and 
open, $g_0 \in U$ and $F_1:  U \rightarrow G$ and $F_2:  U \rightarrow G$ 
are two Lipschitz maps such that $DF_1(g) = DF_2(g)$ for almost all $g \in 
U$ with respect to Haar measure and $F_1(g_0) = 
F_2(g_0)$.  Then $F_1 = F_2$.
\end{fac}

\begin{proof}
Suppose there exists $u \in U$ with $F_1(u) \neq F_2(u)$.  Fix $\epsilon > 0$ such that 
$$d_{CC}(F_1(u), F_2(u)) > \epsilon.$$
Letting $\gamma$ be a piecewise horizontal curve in $U$ joining $g_0$ to $u$ we note that 
there exists $g' \in G$ sufficiently close to the identity such that the 
left translation of $\gamma$ by $g'$ lies in $U$ (which of course implies 
that $g'g_0, g'u \in U$) with $d_{CC}(F_1(g'g_0), F_2(g'u)) > \epsilon$ and 
almost everywhere on this translation, $DF_1 = DF_2$. However, integration 
then implies $$F_1(g' u) F_1(g' g_0)^{-1} = F_2(g' u) F_2(g' g_0)^{-1}.$$  Therefore, we know that
$$d_{CC}(F_1(g' u), F_2(g' u)) = d_{CC}(F_1(g' g_0), F_2(g' g_0)) > \epsilon;$$ 
as $g'$ can be made arbitarily close to the identity this gives us that 
$$\epsilon \leq d_{CC}(F_1(u), F_2(u)) = d_{CC}(F_1(g_0), F_2(g_0)) = 0$$ 
producing a contradiction, so we conclude that $F_1 = F_2$ as desired.
\end{proof}

In fact, because each linear map $\psi$ from the horizontal component of 
$G$ to the horizontal component of $H$ has at most one extension (which 
we call $\tilde{\psi}$) to a Lie 
group homomorphism from $G$ to $H$, we can go one step further and say 
that if $MF$ is equal to the linear map $\psi$ almost everywhere and $g_0 
\in G$, $h_0 \in H$ with $F(g_0) = h_0$ then

$$F(g) = h_0 \tilde{\psi}(g_0^{-1} g)$$ 
for all $g$ where $F(g)$ is defined.

\subsubsection{Weak convergence}\label{sec234}
If a sequence $f_n$ of uniformly Lipschitz functions on a bounded 
Euclidean region converges uniformly to some function $f$ then $f$ 
is Lipschitz, and moreover the Jacobians $Jf_n$ converge weakly in $L^2$ 
to the Jacobian of $f$.  

In other words, we have the following fact:  
\begin{fac}
Let $U \subset \textbf{R}^k$ be a bounded open set and let $\{f_n\}:  U 
\rightarrow \textbf{R}^m$ be a sequence of uniformly Lipschitz functions 
which converges uniformly to the function $f:  U \rightarrow 
\textbf{R}^m$.  If $g:  U \rightarrow \textbf{R}$ is an $L^2$ function and $D$ represents partial 
differentiation with respect to a fixed vector in $\textbf{R}^k$ then
$$\int_U (Df_n) g \rightarrow \int_U (Df) g,$$
where the integrals are with respect to Lebesgue measure and the 
derivatives in question are defined almost everywhere.
\end{fac}

As will be stated shortly, Fact 2.9 generalizes to Heisenberg groups when the map 
$MF$ induced by the Pansu differential (see the definitions section) is 
used in place of the Jacobian.  In particular, one notes that because $MF$ consists of derivatives 
of horizontal components of $F$ with respect to horizontal tangent 
vectors, $MF$ can be viewed as an array of horizontal derivatives 
of real-valued Lipschitz functions (after postcomposing with the 
appropriate coordinate functions).  Then, the weak convergence in 
question is the following fact:

\begin{fac}
Let $U \subset H_k$ be a bounded open set and let $\{f_n\}:  U 
\rightarrow H_m$ be a sequence of uniformly Lipschitz functions which 
converges uniformly to the function $f:  U \rightarrow H_m$.  If $g:  U 
\rightarrow \textbf{R}$ is an $L^2$ function (with respect to Haar measure) and $D$ 
represents partial differentiation with respect to a fixed left-invariant 
horizontal vector field in $H_k$ then $$\int_U (Df_n) g \rightarrow 
\int_U (Df) g,$$ where the integrals are with respect to Haar measure 
and the derivatives in question are defined almost everywhere.
\end{fac}

Facts 2.9 and 2.10 have the same classical proof, whose sketch is below:

\begin{proof}
(Sketch) Approximate $g$ by a sufficiently smooth test function with 
compact support and integrate by parts.
\end{proof}

\subsubsection{Lipschitz extension}\label{sec235}
If $A$ is a subset of the unit cube of $\textbf{R}^n$ and $f$ is a Lipschitz function from 
$A$ to some Euclidean space, then $f$ can be extended to a Lipschitz 
function on the entire unit cube (or, in fact, to all of $\textbf{R}^n$ 
for that matter).  It is not known whether this extension property also holds for 
maps from a subset of a Heisenberg group $G$ into the same group $G$ (see 
\cite{Ba} and \cite{Br}), and for that reason we assume the Lipschitz map 
in Corollary 2.17 below is defined on an open subset of $G$.  It has been shown recently in \cite{Ba} that this extension property does not hold for maps from $\textbf{R}^k$ to $H_n$ with $n < k$.  Also, \cite{RW} shows that the property does not hold for maps from $\textbf{R}^k$ to any jet space on $H_n$ whenever $n < k$. However, 
this property does hold for maps from any Carnot group to any $\textbf{R}^k$.  It also holds for maps from $\textbf{R}^2$ to $H_n$ for $n \geq 2$, as was shown in \cite{F} and \cite{M}.  More generally, based on recent results in \cite{WY} the Lipschitz extension property holds for maps from any set with Assouad-Nagata dimension less than or equal to $n$ to any jet space group on $\textbf{R}^n$.  Notably, this implies the Lipschitz extension property for maps from $H_k$ to $H_{2k + 1}$.

 \subsection{General Carnot groups}\label{sec24}

In this subsection we seek to explain how the constructions performed in Section 2.3 on the Heisenberg group can be generalized to work on other Carnot groups.

First we note that the decomposition in Section 2.3.1 relied on passing from the Heisenberg group to a discrete version.  To this end, we introduce the following definition:

\begin{defi}
Let $G$ be a Carnot group whose Lie algebra $\mathfrak{g}$ is graded as 
$$\mathfrak{g} = \oplus_{j = 1}^d \mathfrak{g}_j.$$
Write $m_j$ as the vector space dimension of $\mathfrak{g}_j$ for $1 \leq j \leq d$.
We say that $G$ is $\textbf{discretizable}$ if for $1 \leq j \leq d$ there exist collections
$$\{X_{(j, i)}\}_{i = 1}^{m_j} \in \mathfrak{g}_j,$$ 
$$\{g_{(j, i)}\}_{i = 1}^{m_j} \in G,$$ 
and subgroups
$$H_j \leq G$$
such that 
$$\{X_{(j, i)}\}_{i = 1}^{m_j} \textit{ spans } \mathfrak{g}_j \textit{ as a vector space,}$$
$$\exp(X_{(j, i)}) = g_{(j, i)},$$
$$H_j = \langle \{g_{(j', i)}\}_{1 \leq i \leq m_{j'}, j \leq j' \leq d} \rangle,$$ 
and writing $G' = \langle \{g_{(1, i)}\}_{i = 1}^{m_1} \rangle$ and $G_j = \langle \{\exp(\mathfrak{g}_{j'})\}_{j' \geq j} \rangle,$
$$G' \textit{ is a discrete subgroup with } G' \cap G_j = H_j.$$
 
\noindent In this setting, we say that $G'$ is the \textbf{discretization} of $G$.

\end{defi}

Examples of discretizable Carnot groups include Heisenberg groups, Euclidean spaces, and jet spaces.  For example, we can take the discrete Heisenberg groups as the discretization of the Heisenberg groups.

If $G$ is discretizable, then the method in \cite{Ch} can be followed as in Section 2.3.1 to create a dyadic 
decomposition, with $B_0$ now defined to be the discretization $G'$.  Although the scaling constant used to create $B_n$ from $B_0$ (which was $10^{-n}$ in the case of Heisenberg groups) depends on the specific Carnot group itself (in particular, it depends on the relationship between the 
coordinates of an arbitrary point $g$ and $d_{CC}(g, 0)$; cf Theorem 
2.10 in \cite{Mo}), the procedure for Heisenberg groups can otherwise be 
copied exactly to create a dyadic decomposition for $G$ into dyadic ``cubes''.   Because the base cube for our construction will still 
be a cube based at the origin of scale zero, we can still refer to it as 
$Q(0, 0)$.  With our new dyadic decomposition in hand, we can also copy the construction of the $C_\beta$ and $C'_\beta$ in Section 2.3.2 in the setting of our discretizable group $G$.

Finally, we observe that the results in Sections 2.3.3 and 2.3.4 (which involved differentiability and weak convergence) used no properties specific to Heisenberg groups.  Therefore, the results in Section 2.3.3 and 2.3.4 carry over just as well to maps from one Carnot group to another.  Actually, Fact 2.8 in Section 2.3.3 was already stated and proved in terms of Carnot groups.

\subsection{Proof of main theorem}\label{sec25}
In what follows, $H^k$ and $h^k$ shall refer to Hausdorff $k$-dimensional 
measure and Hausdorff $k$-dimensional content, respectively (both of 
which we define with respect to Carnot-Carathe\'odory 
distance).

Our goal is to prove the following theorem.
\begin{teo}
Let $G$ be a discretizable Carnot group of homogeneous dimension $k$ and let $H$ be another 
Carnot group.  Suppose $F:  Q(0, 0) \subset G \rightarrow H$ is Lipschitz.  
If $\delta > 0$, there exists a positive 
integer $N$ and subsets $Z, X_1, \dots, X_N$ of $Q(0, 0)$ such that 
$$h^k(F(Z)) < \delta,$$ $$Z \cup X_1 \cup \dots \cup X_N = Q(0, 0),$$ and $F$ 
is biLipschitz on each $X_i$.  Furthermore, $N$ and the biLipschitz 
coefficients of the $F|X_i$ depend only on the groups 
$G$ and $H$, $\delta$, and the Lipschitz coefficient of $F$, and not on the map $F$ itself.
 \end{teo}

Before beginning the proof, we shall introduce two notions of nearness.

\begin{defi}
Suppose $Q(x, \alpha)$ and $Q(y, \alpha)$ are elements of the 
decomposition from \ref{sec231} of a discretizable Carnot group into cubes 
of the same scale.  We say that $Q(x, \alpha)$ and $Q(y, \alpha)$ 
are \textbf{adjacent} if the distance from $Q(x, \alpha)$ to $Q(y, 
\alpha)$ is bounded above by the diameter of $Q(x, \alpha)$.
\end{defi}
\noindent Note that two coincident cubes of the same scale are considered adjacent.

\begin{defi}

Suppose $Q(x, \alpha)$ and $Q(y, \alpha)$ are elements of the 
decomposition of a discretizable Carnot group into cubes of the same scale.  
We say that $Q(x, \alpha)$ and $Q(y, \alpha)$ are 
\textbf{semi-adjacent} if $Q(x, \alpha)$ and $Q(y, \alpha)$ are not 
adjacent and the parents of $Q(x, \alpha)$ and $Q(y, \alpha)$ are not 
adjacent, but the grandparents of $Q(x, \alpha)$ and $Q(y, \alpha)$ are 
adjacent.
\end{defi}

Turning to the proof of Theorem 2.12, we begin by establishing some further notation and normalizations.  

Let $E$ be the ratio of the diameter of an arbitrary ``cube'' to the diameter of one of its ``children'' using 
Carnot-Carath\'eodory distance.  For example, if $G$ is a Heisenberg group (using exactly the 
``cube'' decomposition from Section 2.3.1), then $E = 10$.

Using the Carnot-Carath\'eodory distance, we set
$$\theta = \textnormal{diam}(Q(0, 0)).$$

Also, we let $0 < L_1 < L_2 < \infty$ be constants such that if $Q$ and $Q'$ are semi-adjacent, there exist constants $0 < L_1 < L_2 < \infty$ 
$$L_1 \textnormal{diam}(Q) < d(Q, Q') < L_2 \textnormal{diam}(Q).$$
\noindent
We note that $L_1$ and $L_2$ only depend on $G$, not $Q$ or $Q'$.

In addition, we may assume that $F$ is 1-Lipschitz and that there exists $\eta > 0$ such that $F$ is defined on 
the dilation $\delta_{1 + \eta}Q(0, 0)$. For convenience 
we scale Hausdorff measure so that $|Q(0, 0)| = 1$ where $|S|$ denotes the Hausdorff measure of $S$.

Finally, we let $W$ be a positive integer 
such that every cube $Q'$ of scale $W - 10$ such that $Q' \cap Q(0, 0) 
\neq 0$ satisfies $Q' \subset \delta_{1 + \eta}Q(0, 0)$.  Throughout this proof, we will be focusing primarily on subcubes of $Q(0, 0)$ of scale at least $W$.

With our notation and normalizations set up, we prove the following proposition, which provides a 
partial wavelet decomposition of the linear map $MF$ induced by the Pansu 
differential $DF$ of $F$.

\begin{prop}
Suppose $1 \geq \epsilon > 0$.  There exists $n, C > 0$ 
such that if 
$\alpha \geq 
W$ and $Q := Q(a, \alpha)$ and $Q' := Q(b, \alpha)$ are semi-adjacent cubes with  
\begin{equation}\label{eqn1}
h^k(F(Q)) > \epsilon E^{-k\alpha}
\end{equation}
 and 
\begin{equation}\label{eqn2}
h^k(F(Q')) > \epsilon E^{-k\alpha}
\end{equation}
but
\begin{equation}\label{eqn3}
d_{CC}(F(Q), F(Q')) \leq \frac{1}{2} \epsilon L_1 \theta E^{-\alpha}
\end{equation}

\noindent
then there exists $\beta \in 
[\alpha - 4, \alpha + n]$ and $f_{Q, Q'} \in C'_\beta$
and integers $i, j$ such that

\begin{equation}\label{eqn4}
\frac{|\langle (MF)_{i,j}, f_{Q, Q'} \rangle |}{|\langle f_{Q, Q'}, f_{Q, 
Q'} \rangle |}  \geq C |Q|^{.5}.
\end{equation}

\noindent where $C'_\beta$ is the space defined in Section 2.3.2 and $MF$ 
is the matrix of horizontal components of the Pansu differential $DF$.

Further, $C$ only depends on $G$, $H$, and $\epsilon$ (and, in 
particular, not on the specific choice of $F$).
\end{prop}

Also, the inner product in (4) is taken with respect to $L^2(G)$; it equals
$$\int_G (MF)_{i,j} f_{Q, Q'} d\mu$$ where $\mu$ is Haar measure on $G$ scaled so that $\mu(Q(0, 0)) = 1$.

We also note that the number of possible 
candidates for $f_{Q, Q'}$ for a given $Q$ is 
uniformly bounded, with a bound that depends only on the specific groups 
$G$ and $H$.

\vspace{12 pt}
\begin{proof}  Assume the contrary.  Then, for each $n$ there exists 
a 1-Lipschitz map $F_n$ and semi-adjacent cubes $Q(a_n, \alpha_n)$ and 
$Q(b_n, \alpha_n)$ such that

$$h^k(F_n(Q(a_n, \alpha_n))) > \epsilon E^{-k \alpha_n},$$

$$h^k(F_n(Q(b_n, \alpha_n))) > \epsilon E^{-k \alpha_n},$$

$$d_{CC}(F_n(Q(a_n, \alpha_n)), F_n(Q(b_n, \alpha_n))) < \frac{1}{2} \epsilon L_1 \theta 
E^{-\alpha_n}, $$

\noindent and 

$$\int_{Q(0, 0)} \psi f \leq 2^{-n} |Q(a_n, \alpha_n)|^{.5} 
||f||^2_{L^2(Q(0, 0))}$$

\noindent whenever $\psi$ is a matrix entry of $MF_n$
$$\beta \in [\alpha_n - 4, \alpha_n + n], \textnormal{ and }$$
$$f \in C'_\beta.$$ 
  By rescaling and translating we may suppose $$Q(a_n, \alpha_n) = Q(a, 
\alpha)$$ for all $n$ and by passing to a subsequence we suppose $$Q(b_n, \alpha_n) = Q(b, \alpha)$$ for all $n$.  Further, the Arzel\`a-Ascoli theorem lets us pass to another subsequence such that 
$F_n$ converges uniformly on $Q(0, 0)$ to some Lipschitz map $F$.  
Moreover, by translation (we can do this because of the expanded $C'$ 
families) we can suppose $Q(a, \alpha)$ and $Q(b, \alpha)$ have the same 
great-great-grandparent $Q(z, \alpha - 4)$.  By weak-star convergence, the 
restriction of each component of $MF$ to $Q(z, \alpha - 4)$ is 
orthogonal to $C_\beta$ for $\beta > \alpha - 4$ which implies that $MF$ 
is constant almost everywhere on $Q(z, \alpha - 4)$.  From this, the 
discussion in Section 2.3.3 lets us conclude that there exists a Lie group 
homomorphism $\phi$ such that $DF = \phi$ on $Q(z, \alpha - 4)$ and 
further, there exist elements $g_0 \in G, h_0 \in H$ such that
\begin{equation}
F(g) = h_0 \phi(g_0^{-1} g)
\end{equation}
 for all $g \in Q(z, \alpha - 4)$.  Further, $$h^k(F(Q(a, \alpha)) \geq \liminf_n h^k(F_n(Q(a, 
\alpha))) \geq \epsilon E^{-k \alpha}$$ 
because $F_n(Q(a, \alpha))$ is eventually contained in an arbitrarily 
small neighborhood of the closure $\bar{F}(Q(a, \alpha))$; such a 
neighborhood can have Hausdorff content arbitrarily close to $h^k(F(Q(a, 
\alpha)))$.

Working towards a contradiction, we next define the sequences of points 
$\{X_n\}$ and $\{Y_n\}$ such that 
$$X_n \in Q(a, \alpha), Y_n \in Q(b, \alpha),$$ 
\noindent and 
$$d_{CC}(F_n(X_n), F_n(Y_n)) \leq \frac{1}{2} \epsilon L_1 \theta E^{-\alpha}.$$
By the definition of sequential compactness, there exist points 
$a' \in Q(a, \alpha)$, $b' \in Q(b, \alpha)$ such that $$d_{CC}(F(a'), F(b')) \leq \frac{1}{2} \epsilon L_1 \theta E^{-\alpha}.$$

However, because $Q(a, \alpha)$ and $Q(b, \alpha)$ are semi-adjacent,
$$d_{CC}(a', b') \geq L_1 \theta E^{-\alpha}.$$  
Therefore, since (5) implies that the Pansu differential $DF$ of $F$ 
is defined everywhere and, in fact, is constant, 
the image of the Pansu differential $DF$ of $F$ in the direction of the 
tangent vector from $a'$ 
to $b'$ has 
magnitude at most $\frac{1}{2} \epsilon$.  As $F$ is Lipschitz with 
coefficient $1$, this 
implies that
\begin{equation}
h^k(F(Q(a, \alpha))) \leq |F(Q(a, \alpha))| \leq \frac{1}{2} \epsilon E^{-k 
\alpha}.
\end{equation}
\noindent The first inequality in (6) follows immediately from the fact
that Hausdorff content is bounded above by Hausdorff measure.  The second 
inequality is a direct consequence of the change-of-variables formula for 
Carnot groups (cf the proof of Theorem 7 of \cite{Vo}, which can be 
directly adapted to this case).

As (6) contradicts our hypotheses, the proposition follows.
\end{proof}

Armed with this proposition, our next goal is to show that a sufficiently 
large portion of our domain lies in finitely many such semi-adjacent 
pairs.

\begin{prop}
Let $\Omega$ be the set of all pairs of cubes which satisfy the hypotheses 
of Proposition 2.15 and let
$$\phi(x) = \#\{\omega = (Q, Q') \in \Omega: x \in Q \cup Q'\}.$$

\noindent Suppose $N > 0$; then there exists a constant $K'$ 
depending only $G$, $H$, and $\epsilon$ such that 

$$|\{x:  \phi(x) \geq N\}| \leq K' N^{-1}.$$
\end{prop}

\begin{proof}  If $(Q, Q') \in \Omega$, Proposition 2.15 gives us a wavelet function 
$f_{Q, Q'}$ corresponding to $(Q, Q')$ such that the projection of $MF$ onto $f_{Q, Q'}$ had 
$L^2$ magnitude at least  $C \epsilon |Q|^{1/2}$.  However, only a bounded number of pairs of cubes can be assigned a given wavelet function in this way.  This is because of the control that Proposition 2.15 gives to both the scale and support of $f_{Q, Q'}$ in terms of the scale and location of $Q$.  Now, we seek to show that 

$$1 \succeq \Sigma_{(Q, Q') \in \Omega} |Q|$$

\noindent where the implied multiplicative constant only depends on $G$, $H$, 
and $\epsilon$.

Because $F$ is $1$-Lipschitz, we can replace the constant $1$ on the left hand side with $||MF||_2^2$.
Next, for any specific pair $(Q, Q')$ in our sum, we let $\pi_{(Q, Q')} (MF)$ be the orthogonal projection of $MF$ onto $f_{Q, Q'}$.  By Proposition 2.15,
$$||\pi_{(Q, Q')}(MF)||_2 \geq C \epsilon |Q|^{\frac{1}{2}};$$ in other words,
$$\int |\pi_{(Q, Q')}(MF)|^2 \geq C^2 \epsilon^{2} |Q|.$$
Summing this over $\Omega$ gives us indeed that 
$$1 \succeq ||MF||_2^2 \succeq \Sigma_{(Q, Q') \in \Omega} |Q|$$
because the $f_{Q, Q'}$ are approximately orthogonal and a given wavelet function can only appear in the sum a bounded number of times.  However, 
$$\int \phi = \Sigma_{(Q, Q') \in \Omega} |Q|,$$
so Chebyshev's inequality therefore tells us that  
$$S_N = \{x:  \phi(x) \geq N\}$$
\noindent has
$$|S_N| \preceq N^{-1},$$ 
which proves the proposition.
\end{proof}

\begin{proof}[Proof of theorem]  We complete the theorem through an 
infinite series of iterations as in \cite{J}.  This process is divided 
into stages (indexed by $\alpha \geq 0$); at stage $\alpha$ we assign 
each point $x$ of each subcube of $Q(0, 0)$ of scale $\alpha$ a 
label $x_\alpha$, i.e. a finite string of zeroes and ones, such that every 
point in a fixed cube of scale $\alpha$ has the same label.  

  At stage $0$ we apply a leading digit of 0 to every point in the 
base cube.  In other words, for each $x \in Q(0, 0)$, we set $x_0 = 0.$  
Also, we define $Z_0 = \emptyset$ for future reference.

For $0 < \alpha$, we begin by defining the garbage set $Z_\alpha$ by 
letting $S_\alpha$ be the collection of all cubes $Q$ of scale $\alpha + 
W$ 
such that $$|F(Q)| \leq \delta E^{-k (\alpha + W)}$$ 

\noindent and set $Z_\alpha = S_\alpha \cup Z_{\alpha - 1}$.

Next, we run through each pair of cubes at scale $\alpha + W$ which lie in 
$Q(0, 0) \backslash Z_\alpha$ and which satisfy the hypotheses of 
Proposition 2.16 with $\epsilon = .01\delta$.  Supposing that there 
are $n_\alpha$ such pairs $(Q_1, Q'_1), \dots, (Q_{n_\alpha}, 
Q'_{n_\alpha})$, we will inductively define the labels 
$x_{(\alpha, m)}$ for $m = 0, 1, \dots, n_\alpha$ as follows:

First, $x_{(\alpha, 0)} = x_{\alpha - 1}$ for each $x \in Q(0, 0) 
\backslash Z_\alpha$.
Then, for $m > 0$ we define $x_{(\alpha, m)} = x_{(\alpha, m - 1)}$ for $x 
\notin Q_m \cup Q'_m$.  We note that $x_{(\alpha, m - 1)}$, when viewed as 
a function on $Q(0, 0) \backslash Z_\alpha$, is constant at 
a value (call it $z_1$, and let $y_1$ be its length) on $Q_m$ and at a 
possibly different value (call it $z_2$, and let $y_2$ be its length) on 
$Q'_m$; without loss of generality we may assume that $y_1 \geq y_2$.
There are several cases to consider:

I) If $y_1 = y_2$ and $z_1 \neq z_2$ we simply define 
$x_{(\alpha, m)} = x_{(\alpha, m - 1)}$ on both $Q_m$ and $Q'_m$.

II) If $y_1 = y_2$ and $z_1 = z_2$ we then let $x_{(\alpha, m)}$ be equal 
to the string created by adding a $0$ to the end of $x_{(\alpha, m - 1)}$ 
on $Q_m$ and the string created by adding a $1$ to the end of $x_{(\alpha, 
m - 1)}$ on $Q'_m$.

III) If $y_1 > y_2$ and $z_2$ IS NOT the first $y_2$ digits of $z_1$ we 
simply define $x_{(\alpha, m)} = x_{(\alpha, m - 1)}$ on both $Q_m$ and $Q'_m$.

IV) If $y_1 > y_2$ and $z_2$ IS the first $y_2$ digits of $z_1$, we let 
define $x_{(\alpha, m)} = x_{(\alpha, m - 1)}$ on $Q_m$; on $Q'_m$ we let 
$y'$ be the element of $\{0, 1\}$ that IS NOT the ($y_2 + 1$)th digit of 
$z_1$ and define $x_{(\alpha, m)}$ on $Q'_m$ to be the string created by 
adding $y'$ to the end of $x_{(\alpha, m - 1)}$.

Once we have finished this process for each cube, we define $x_\alpha = 
x_{(\alpha, n_\alpha)}$ on $Q(0, 0) \backslash Z_\alpha$.

Now, defining $Y_n$ to be the set of all points $x$ such that $x_\alpha$ 
has length at least $n$ for some $\alpha$, we conclude from Proposition 
2.16 that there exists $N$ such that 
$$|\{x \in Q(0, 0) \backslash \bigcup_\alpha Z_\alpha:  x \in Y_N\}| < .01 
\delta;$$

\noindent we now define the set $Z = \bigcup_\alpha Z_\alpha \cup Y_N$.

If $x \in Q(0, 0) \backslash Z$, then the sequence $\{x_\alpha\}$ is 
eventually constant; denote its limiting value by $x_\infty$.  Note that 
there are are at most $2^N$ possible values of $x_\infty$ as 
there are at most $2^n$ strings of length $n$.   

We finish by setting $$X_w = \{x \in Q(0, 0) \backslash Z:  x_\infty = 
w\}$$ 
whenever $w$ is a string of zeroes and ones of length less than $N$.
For each such $w$, $F|X_w$ must be biLipschitz (if not, there exist $x_1, 
x_2 \in X_w$ and a pair of cubes $(Q, Q')$ satisfying the hypotheses of 
Proposition 2.15 such that $x_1 \in Q, x_2 \in Q'$, contradicting the 
definition of $X_w$), proving the theorem.
\end{proof}

\subsection{Consequences}\label{sec26}

\begin{co}
Suppose $A$ is an open subset of a discretizable Carnot group $G$ (with homogeneous 
dimension $k$), $H$ is another Carnot group, and $F: A \rightarrow H$ is 
Lipschitz, and $H^k(F(A)) > 0$.  Then there exists 
a subset $B \subset A$ of positive $k$-dimensional Hausdorff measure such 
that $F$ restricted to $B$ is biLipschitz.  
 \end{co}

\begin{proof}  
We can express $A$ as a countable union of translates and 
dilates of the base cube $Q(0, 0)$; by countable additivity of Hausdorff 
measure one of these cubes, which we call $C$, is sent by $F$ to a set 
$F(C)$ with $H^k(F(C)) > 0$.  By rescaling we can 
suppose $C$ is the base cube $Q(0, 0)$.  
The previous theorem divides this cube into the union of a 'garbage' set 
$Z$  (consisting of those cubes whose image has measure too small, as well 
as those cubes which are in too many bad pairs), where $F(Z)$ can be taken 
to be arbitrarily small (say, with $h^k(F(Z)) < .5 h^k(F(A))$) and a 
finite union of sets $F_j$ such that $F|F_j$ is 
biLipschitz for each $j$.  As 
$H^k(F(\cup_j F_j)) > 0$, there exists some $j$ where $|F_j| > 0$ and we 
let $B = F_j$.
\end{proof}

If one assumed that $H^k(A) < \infty$, looking closely at the shape of $A$ would allow us to conclude above that the measure of $B$ and the biLipschitz constant of $F$ would depend only on 
$G$, $H$, $A$, the Lipschitz coefficient of $F$, and the $k$ 
dimensional Hausdorff content of $F(A)$.

Restricting attention to the first Heisenberg group $H_1$, we use this 
corollary to show that if we only consider maps whose domains are open, 
two questions from \cite{H} are equivalent.  To 
begin we need two more definitions.

\begin{defi}
Suppose $Q_1$ and $Q_2 $ are metric spaces with Hausdorff dimension $k$.  
We say that $Q_1$ \textbf{looks down on} $Q_2$ if there exists a Lipschitz function 
$f$ from some subset of $Q_1$ to $Q_2$ such that the image of $f$ has 
nonzero Hausdorff $k$-measure.
\end{defi}

\begin{defi}
Suppose $Q$ is a metric space with Hausdorff dimension $k$.  We say that 
$Q$ is \textbf{minimal in looking down} if whenever $Q'$ is a metric space 
with Hausdorff dimension $k$ such that $Q$ looks down on $Q'$, $Q'$ also 
looks down on $Q$.
 \end{defi}
Note that the above definition is formulated differently from the definition given in \cite{H}.

Question 22 in \cite{H} asks whether the first Heisenberg group is minimal in looking 
down and Question 24 asks if every Lipschitz map from $H_1$ 
to a metric space with nontrivial Hausdorff 4-measure is biLipschitz on 
some subset with positive Hausdorff 4-measure.  

Clearly 24 implies 22.  However, we now know from the corollary that 22 
implies 24 when only looking at maps from open sets.  This is true 
because (assuming $H_1$ is minimal in looking down) if $F :  E 
\subset H_1 \rightarrow X$ is Lipschitz and $H^4(F(E)) > 0$ then, letting 
$G: X \rightarrow H_1$ be another Lipschitz map with $H^4(G(X)) > 0$
(and supposing, by restricting images, that $X = F(E)$), $G \circ F$ 
satisfies the conditions of the corollary and therefore is biLipschitz on 
some subset $E' \subset E$ with $|E'| > 0$.  On this 
set, we therefore 
have that $F$ is invertible with inverse $(G \circ F)^{-1} \circ G$, 
which is clearly Lipschitz, which therefore implies that $F|E'$ is 
biLipschitz.
Because $F$ was arbitrary, we can conclude that Question 24, when 
restricted to maps defined on open sets, is equivalent to Question 22.

Raanan Schul recently proved a statement corresponding to Question 24 for 
maps where the domain is Euclidean in \cite{Sc}.  In particular, he showed 
that if $F$ is a Lipschitz function from the $k$-dimensional unit cube 
$[0, 1]^k$ into a general metric space, one can decompose $$[0, 1]^k = G 
\cup \bigcup_{j = 1}^n F_j$$ where $F(G)$ has arbitrarily small $k$-dimensional Hausdorff 
content and $F$ is biLipschitz on each of the $F_j$.  The main reason why 
Schul's argument does not generalize to this setting is the dearth of 
rectifiable curves passing through a given point in a general Carnot group.
For example, although the first Heisenberg group has Hausdorff dimension 
4, the space of horizontal tangents to rectifiable curves through a given 
point in that group has dimension two.

We finish this section by discussing the question of Jones-style 
decompositions for Lipschitz maps on Carnot groups.  Just as in the work 
of Peter Jones in \cite{J}, my argument for the main theorem actually 
implies the following stronger statement:

\begin{co}
Suppose $U$ is a bounded open subset of a discretizable Carnot group $G$ with Hausdorff 
dimension $Q$, $H$ is another Carnot group, $F: U \rightarrow H$ is 
Lipschitz, and $\epsilon > 0$.  Then there exists a finite collection 
$\{A_i\}$ of subsets of $U$ such that each restriction $F|A_i$ is 
biLipschitz and $$h^Q(F(U \backslash \cup_i A_i)) < \epsilon.$$
\end{co}

\noindent For unbounded open subsets of discretizable Carnot groups a diagonalization 
argument yields the following.

\begin{co}
Suppose $U$ is an open subset of a discretizable Carnot group $G$ with Hausdorff 
dimension $Q$, $H$ is another Carnot group and $F: U \rightarrow H$ is 
Lipschitz.  Then there exists a countable collection 
$\{A_i\}$ of subsets of $U$ such that each restriction $F|A_i$ is 
biLipschitz and $$h^Q(F(U \backslash \cup_i A_i)) = 0.$$
\end{co}

A natural generalization of the above results is in the setting of sub-Riemannian manifolds, defined below.

\begin{defi}
A \textbf{sub-Riemannian manifold} is a triple $(M, \Delta, g)$ where $M$ 
is a smooth manifold, $\Delta$ is a distribution (i.e. sub-bundle of the 
tangent bundle $TM$) on $M$ which is smooth and satisfies the property 
that for each $p \in M$, $(TM)_p$ is generated as a Lie algebra by 
$\Delta_p$, and $g$ is a smooth section of positive-definite quadratic 
forms on $\Delta$ (i.e. $g_p$ defines an inner product on $\Delta_p$ which 
varies smoothly in $p$).
\end{defi}

Recall (see \cite{V}) that the set $S$ is said to generate a Lie algebra 
$\mathfrak{g}$ if the set of finite Lie brackets of elements of $S$ spans 
$\mathfrak{g}$ as a vector space.

We shall consider $M$ to be naturally equipped with a metric $d_{CC}$ 
defined as follows:  for $x, y \in M$,

$$d_{CC}(x, y) = \inf_{\gamma \in \Gamma_{x, y}} \int_0^1 
\sqrt{g(\gamma'(t), \gamma'(t))} dt$$

\noindent where $\Gamma_{x, y}$ is the family of all curves $$\gamma: [0, 
1] \rightarrow M$$ with $\gamma(0) = x$, $\gamma(1) = y$, and $\gamma'(t) \in 
\Delta_{\gamma(t)}$ for all $t$.

Now, suppose $M$ and $N$ are sub-Riemannian manifolds such that $M$ is locally biLipschitz equivalent to a discretizable Carnot group $G$ and $N$ is locally biLipschitz equivalent to a Carnot group $H$.  Then Corollary 2.21 still holds if $G$ is replaced by $M$ and $H$ is replaced by $N$.

For example, $M$ and $N$ could both be ordinary Riemannian manifolds.  Because Riemannian manifolds are locally 
biLipschitz equivalent to Euclidean spaces, where we have all five properties from Section 2, we can consider arbitrary subsets of $M$ instead of just open subsets.  Thus we have the following corollary:  
if $M$ is a Riemannian manifold, $A \subset M$ has Hausdorff dimension $k$, $N$ 
is another Riemannian manifold, and $F: A \rightarrow N$ is Lipschitz with $H^k(F(A)) > 0$, then 
there exists a subset $B \subset A$ with $H^k(B) > 0$ such that $f|B$ is 
biLipschitz.  

Note that not all sub-Riemannian manifolds are locally biLipschitz 
equivalent to Euclidean spaces, and hence we cannot replace $G$ and $H$ by arbitrary sub-Riemannian manifolds in Corollary 2.21.  In particular, we will show in Section 4.2 that Corollary 2.21 becomes false if $G$ and $H$ are replaced by the Grushin plane and the Euclidean plane, respectively.

\section{Hausdorff Dimension of Lipschitz Images}\label{sec3}

We begin by observing the following corollary of the results in Section 2.

\begin{co}

Assume $A$ is an open subset of a discretizable Carnot group $G$ with 
homogeneous dimension $k$, assume $H$ is another Carnot group, and let $f: A 
\rightarrow H$ be a Lipschitz map such that $H^k(f(A)) > 0$.  Then there exists an injective Lie group 
homomorphism from $G$ to $H$.

\end{co}

\begin{proof}  By the preceding results, $f$ is biLipschitz on some 
$B \subset A$ with positive $k$-dimensional Hausdorff measure.  Then the 
Pansu differential of $f$ at any Lebesgue point of $B$ gives the desired homomorphism.
\end{proof}

Because the converse of this result is trivial (the Lie group homomorphism 
in question is locally Lipschitz), Corollary 3.1 reduces the 
question of whether one Carnot group `looks down' on another to a 
question about the groups' Lie algebras.

An easy consequence of Corollary 3.1 is that if $G$ 
is a discretizable non-abelian Carnot group with homogeneous dimension $k$ and $U \subset G$ then 
every Lipschitz image of $U$ in any Euclidean space has zero $k$-dimensional Hausdorff measure.  This follows because there are no injective group homomorphisms from a 
nonabelian group to an abelian group.  In fact, for this consequence we need not assume $U$ is open here because the image space, Euclidean space, has the Lipschitz extension property.

Despite having Hausdorff measure $k$-measure zero, the Lipschitz image of $U$ in $\textbf{R}^k$ can still be 
quite large.  For example, we have the following theorem, which answers a 
question asked by Enrico Le Donne (cf \cite{L}):

\begin{teo}

Suppose that $G$ is a discretizable Carnot group with homogeneous dimension $k$, and 
let $\epsilon > 0$.  There exists a bounded open $U \subset G$ and a 
Lipschitz map $F:  U \rightarrow \textbf{R}^k$ such that $H^{k - 
\epsilon}(F(U)) > 0$.

\end{teo}

\begin{proof}  As in our results in Section 2, we illustrate the case $G 
= H^1$ in detail and remark that the construction is analogous for the general case.  The construction is based on the procedure from \cite{K}.

\noindent We begin by setting $$\gamma = 16^{\frac{1}{\epsilon - 4}}$$ which tells 
us that $$\gamma < \frac{1}{2}$$ and $$\log_{\gamma^{-1}} 16 = 4 - 
\epsilon.$$

\noindent We next fix $\beta \in [\gamma, \frac{1}{2})$ and define $$\lambda = 
\frac{20}{\frac{1}{4} - \beta^2};$$ in particular,
$$\lambda (\frac{1}{4} - \beta^2) = 20 > 10.$$

\noindent With this data, we then set our initial box 
$$I^0 = [-1, 1] \times [-1, 1] \times [-\lambda, \lambda] \subset H^1$$ 
and 
$I^1$ to be the union of the sixteen boxes 
$$(a, b, c) \cdot \delta_\beta I_0$$ where 
$$a \in \{-.5, .5\}, b \in \{-.5, .5\},$$ and $$c \in \{-.75 
\lambda, -.25 \lambda, .25 \lambda, .75 \lambda\}.$$  
We arbitrarily label these boxes $I^1_j$ for $j = 1, \dots, 16$. 

The point of this construction is to find $\eta > 0$ such that

$$d_{CC}(I^1_j, I^1_k) > \eta \textrm{ for } j \neq k$$
and
$$d_{CC}(I^1_j, \delta(I^0)) > \eta \textrm{ for all } j.$$

Clearly, if two of the boxes in $I^1$ have different horizontal 
components, then they are at least $1 - 2 \beta$ apart; similarly, every
box in $I^1$ is at a distance of exactly $.5 - \beta$ away from the 
nearest horizontal edge of $I_0$.

The only issue is vertical distance.  To find the minimum distance between 
a vertical edge of $I^0$ and a box in $I^1$, it suffices to consider a box 
in $I^1$
where $c = -.75 \lambda$ and look at the bottom edge of $I^0$.  Every 
point on the bottom 
edge of such a box has a vertical coordinate which is at
least $$-.75 \lambda - \beta^2 \lambda - 2 \cdot .5 \beta > 
-\lambda + 10 
- 2 = -\lambda + 8.$$
Now, we recall that if $g = (x_1, y_1, 0)$ and $h = (x_2, y_2, 0)$ are 
points in $H^1$ with $x_1, y_1, x_2, y_2 \in [-1, 1]$, then writing the 
product $g^{-1} h$ as $(x_3, y_3, z_3)$ we note that $|z_3| < 2$.  

Consequently, if $p = (p_1, p_2, p_3)$ is a point in $I_1$ and 
$q = (q_1, q_2, -\lambda)$ is a point on the bottom edge of $I_0$, we 
note that the vertical coordinate of $p^{-1} q$ is at most 
$$-(-\lambda + 8) - \lambda + 2 = -6,$$

\noindent which implies that vertical edges of $I^0$ will be separated 
from boxes in $I^1$ by at least 6 units.

Similarly, looking at two boxes in $I^1$ with the same horizontal 
component (e.g. let $A$ be such a box with $c = -.75 \lambda$ and $B$ be 
such a box with $c = -.25 \lambda$), the 
vertical coordinate of points in $A$ are at most $-.5 
\lambda - 8$ 
and the  vertical coordinate of points in $B$ are at least $-.5 
\lambda + 8$.
Therefore, whenever $a \in A$ and $b \in B$, the vertical 
coordinate of $a^{-1} b$ is at least 
$$(\lambda + 8) - (\lambda - 8) - 2 = 14,$$

\noindent implying a separation of 14 between any two such boxes.

In subsequent stages we replace each box of the form $$p \cdot 
\delta_\mu I^0$$ (there are $16^k$ such boxes in stage $k$; at this stage $\mu = \beta^k$) with the 
sixteen boxes 
$$p \cdot \delta_\mu(a, b, c) \cdot \delta_{\beta \mu} \} I^0$$
and call the union of all the boxes produced in stage $k$ $I^k$.

In stage $k$, each box has a label of the form $I^k_{(a_1, \dots, a_k)}$ 
where each $a_i$ ranges from one to sixteen; we extend this process to 
stage $k + 1$ by labeling the subboxes from $I^k_{(a_1, \dots, a_k)}$ as 
$I^{k + 1}_{(a_1, \dots, a_k, v)}$ where $v = 1, 2, \dots, 16$.  The 
intersection of the $I^k$'s, to 
be defined as $I$, is a Cantor set in $H^1$ of dimension 
$$\log_{\beta^{-1}} 16 \geq 4 - \epsilon.$$  Each point $x \in I$ 
has a unique label of the form 
$(a_1, \dots, a_n, \dots)$ where each $a_i$ ranges from one to sixteen 
such that for each $n \in \textbf{N}$, $x \in I^n_{(a_1, \dots, a_n)}$; if 
$v = (a_1, \dots, a_n, \dots)$ and $w = (b_1, \dots, b_n, \dots)$ with $m$ 
being the smallest integer where $a_m \neq b_m$, the distance between the 
points corresponding to $v$ and $w$ is (up to a multiplicative constant 
independent of $m$) equal to $\beta^m$.

Similarly, we set $J^0$ to be the box $[-1, 1]^4$ in Euclidean space 
$\textbf{R}^4$
and $J^1$ to be the union of the sixteen boxes $$(a, b, c, d) + \gamma 
I^0$$ 
where $a, b, c, d$ can each equal $-.5$ or $.5$. We arbitrarily label 
these boxes $J^1_j$ for $j = 1, \dots, 16$. 

The point of this construction is now to find $\eta' > 0$ such that

$$d(J^1_j, J^1_k) > \eta' \textrm{ for } j \neq k$$
and
$$d(J^1_j, \delta(J^0)) > \eta' \textrm{ for all } j.$$

where the distance above is Euclidean.

Clearly, any two of the boxes in $J^1$ are at least $1 - 2 \gamma$ 
apart; similarly, each such box is at a distance of exactly $.5 - \gamma$ 
away from the boundary of $J^0$.  

In subsequent stages we replace the box $$p + \nu J^0$$ with the 
sixteen boxes $$p + \nu ((a, b, c, d) + \gamma J^0)$$ and call the 
union of all boxes produced in stage $k$ $J^k$.  Note that at stage $k$, $\nu = \gamma^k$.

In stage $k$, each box has a label of the form $J^k_{(a_1, \dots, a_k)}$ 
where each $a_i$ ranges from one to sixteen; we extend this process to 
stage $k + 1$ by labeling the subboxes from $J^k_{(a_1, \dots, a_k)}$ as 
$J^{k + 1}_{(a_1, \dots, a_k, v)}$ where $v = 1, 2, \dots, 16$.  The 
intersection of the $J^k$'s, to 
be defined as $J$, is a Cantor set in $\textbf{R}^4$ of dimension 
$$\log_{\gamma^{-1}} 16 = 4 - \epsilon.$$  Each point $x \in J$ 
has a unique label of the form 
$(a_1, \dots, a_n, \dots)$ where each $a_i$ ranges from one to sixteen 
such that for each $n \in \textbf{N}$, $x \in J^n_{(a_1, \dots, a_n)}$; if 
$v = (a_1, \dots, a_n, \dots)$ and $w = (b_1, \dots, b_n, \dots)$ with $m$ 
being the smallest integer where $a_m \neq b_m$, the distance between the 
points corresponding to $v$ and $w$ is (up to a multiplicative constant 
independent of $m$) equal to $\gamma^m$.

We can define a Lipschitz map $F$ from $I^0 \subset H^1$ to $\textbf{R}^4$ 
whose image contains $J$ (and therefore has Hausdorff dimension 
$\log_{\gamma^{-1}}(16)$) via the following three-step process.

\textbf{Step 1}:  Map $I$ to $J$.  This is done by mapping a point in $I$ 
with a label of the form $(a_1, \dots, a_n, \dots)$ to the point with the 
same label in $J$.  By construction, one notes that if $\beta = \gamma$ then 
this map is biLipschitz.

\textbf{Step 2}:  For each ordered $n$-tuple $(a_1, \dots, a_n)$ with each 
$a_i$ in $\{1, \dots, 16\}$ (this includes the zero-tuple, where we would 
be mapping the boundary of $I^0$) we choose a point $p_{(a_1, \dots, 
a_n)}$ 
in $J^n_{(a_1, \dots, a_n)}$ and then send all of the points in the 
boundary of $I^n_{(a_1, \dots, a_n)}$ to $p_{(a_1, \dots, a_n)}$.

\textbf{Step 3}:  The remaining region of $I^0$ consists of sets of the 
form $S^n_{(a_1, \dots, a_n)}$ defined as the set of all points in 
$I^n_{(a_1, \dots, a_n)}$ which do not lie in $I^{n + 1}_{(a_1, \dots, 
a_n, v)}$ for $v = 1, 2, \dots, 16$.  The closure of this region includes 
the boundary of 
$I^n_{(a_1, \dots, a_n)}$ and of $I^{n + 1}_{(a_1, \dots, a_n, v)}$ for $v 
= 1, \dots, 16$.  Fixing $(a_1, \dots, a_n)$ (we may work on each 
$S^n_{(a_1, \dots, a_n)}$ separately) we define the map $f$ from the 
interval $[0, 16]$ to $\textbf{R}^4$ to be a smooth function sending $0$ 
to $p_{(a_1, \dots, a_n)}$ and $v = 1, \dots, 16$ to $p_{(a_1, \dots, 
a_n, v)}$.  We 
can suppose $f$ has Lipschitz norm comparable to $\gamma^n$.  We then 
define $g$ to be a smooth, real-valued, Lipschitz function (with 
Lipschitz coefficient comparable to $\beta^{-n}$) on the closure of 
$S^n_{(a_1, \dots, a_n)}$ which sends the boundary of $I^n_{(a_1, \dots, 
a_n)}$ to $0$ and the boundary of $I^{n + 1}_{(a_1, \dots, a_n, v)}$ to 
$v$.  We can create such a $g$ by the Whitney extension theorem (the 
construction is more straightforward if we do not require smoothness).  
On the closure 
of $S^n_{(a_1, \dots, a_n)}$ (the 
construction merely repeats the existing one on the boundary) set $F = f 
\circ g$; $F|\overline{S^n_{(a_1, \dots, a_n)}}$ has Lipschitz norm 
comparable to $(\frac{\gamma}{\beta})^n$.

Note that if $\gamma < \beta$, $(\frac{\gamma}{\beta})^n$ goes to zero as 
$n$ goes to 
infinity, which means that $F$ is differentiable (in the Pansu sense) at 
each point of $I$ with derivative zero.  Further, by construction $F$ is 
$C^1$ outside of $I$ where the Pansu differential always has rank zero or 
one (and this differential approaches zero as we approach points of $I$); 
in fact, it is locally constant near the boundaries of the relevant cubes 
if we use the Whitney extension, so the construction here is indeed an 
appropriate analogue of \cite{K}.
\end{proof}

In fact, because the constructed map is constant on the boundary of $I^0$, 
nesting appropriately-rescaled examples of this form inside each other 
yield the following corollary.

\begin{co}

Suppose that $G$ is a discretizable Carnot group with homogeneous dimension $k$.  There 
exists a bounded open $U \subset G$ and a Lipschitz map $F:  U 
\rightarrow \textbf{R}^k$ such that $F(U)$ has Hausdorff dimension $k$.

\end{co}

\section{Counterexamples}\label{sec4}

In this section we develop two counterexamples to show why Carnot group 
structure, or something close to it, is necessary for the results of 
the previous two sections.

\subsection{A Space-Filling Curve}\label{sec41}

\begin{teo}
There exists an Ahlfors $2$-regular metric space $X$ and a Lipschitz map 
$F: X \rightarrow \textbf{R}^2$ such that $F(X)$ has positive $2$-dimensional 
Hausdorff measure but $F$ is not biLipschitz on any set of positive 
$2$-dimensional measure.
\end{teo}

\begin{proof}  The function in question will be the space-filling curve 
$F$ from $[0, 1]$ (equipped with the square root distance metric) to the 
unit square in $\textbf{R}^2$ mentioned in Section 7.3 of \cite{Ste}.  
Although this function is a surjective map of spaces with Hausdorff 
dimension 2 and Lipschitz, it is not biLipschitz on any subset with positive
Hausdorff 2-measure.  
To see this, suppose that the space-filling curve $F$ is biLipschitz on a 
set $A$ with $H^2(A) > 0$.  As $F(A)$ has 
positive Lebesgue 
measure, it contains a point $x$ of Lebesgue density one. Letting 
$\epsilon > 0$ there exists $\delta > 0$ such that $$|B(x; \delta) \cap 
F(A)| > (1 - \epsilon) |B(x; \delta)|.$$  Writing out the binary expansion 
of the components of $x$ and of $\delta$, $B(x; \delta)$ contains a 
dyadic cube $Q$ of side at least $.1 \delta$; as $$\epsilon |B(x; 
\delta)| \leq 1000 \epsilon |Q|, |Q \cap F(A)| > (1 - 1000\epsilon) 
|Q|.$$ 

\noindent As $F$ is measure-preserving, letting $J$ be the preimage of $Q$ 
we conclude
$$|J \cap A| > (1 - 1000\epsilon) |J|.$$   By rescaling and translating 
we can suppose $F$ is therefore biLipschitz on a set $A$ of Hausdorff 
2-measure arbitrarily close to $1$ 
(although the rescaled $F$ is not identical to our space-filling curve, 
it preserves all the relevant properties, such as being Lipschitz in the 
appropriate metric, measure-preserving, and sending a pair of points whose 
'square root' distance is at least $\frac{1}{2}$ to the same point).

Let $x$, $x'$ be two points which are at least $\frac{1}{4}$ apart in 
Euclidean distance (and therefore $\frac{1}{2}$ away with respect to square root 
distance) such that $F(x) = F(x')$.  We can suppose that $y, y' \in A$ 
are arbitrarily close to $x$, $x'$ respectively; therefore, $|y - y'| 
\geq \frac{1}{4}$; however, $$|F(y) - F(y')| \leq |F(x) - F(y)| + |F(x') - 
F(y')|$$ which can be made arbitrarily small by the Lipschitz property 
(all distances use the square root metric in the domain and the Euclidean 
metric in the image) showing that $F$ cannot be biLipschitz on $A$ with 
any coefficient.
\end{proof}

In this example, the third and fourth properties (involving 
differentiability) from Section 2.3 fail.  This suggests that some notion 
of differentiability is necessary for the results in \cite{J} to extend 
to other spaces. 

\subsection{The Grushin Plane}\label{sec42}

\begin{teo}
There exists a 2-dimensional sub-Riemannian manifold $M$ with Hausdorff 
dimension 2, an open $U \subset M$, and a Lipschitz map $$F: U 
\rightarrow{\textbf{R}^2}$$ which is not decomposable in the following 
sense:  There does not exist a countable collection $\{A_i\}$ of sets such that 
$$H^2(F(U\backslash \cup_i A_i)) = 0$$ and $F|A_i$ is biLipschitz for 
each $i$.
 \end{teo}

\begin{proof}  We use the Grushin plane $M$ as our sub-Riemannian 
manifold.

To construct the Grushin plane we define a Riemannian metric on the 
following region of $\textbf{R}^2$:  $\{(x, y):  y \neq 0\}$.

This metric is defined as $ds^2 = dx^2 + x^{-2} dy^2$.  We then use this 
metric to induce a geodesic structure on all of $\textbf{R}^2$, where a 
rectifiable curve must have horizontal tangent at each point that it 
crosses the $y$-axis.

One can observe that off of the vertical axis, the Grushin plane is locally 
biLipschitz to Euclidean space (but with a constant that blows up as we 
get closer to the axis).  However, the distance between two points on the 
vertical axis is proportional to the square root of their Euclidean distance.

In other words, the Grushin plane is a union of a 
(disconnected) Riemannian manifold and a line of Hausdorff dimension two, 
making it a sub-Riemannian manifold of both Euclidean and Hausdorff 
dimension two.

To construct our counterexample, we consider an open neighborhood of the 
segment $S$ joining $(0, 0)$ to $(0, 1)$, say:  $U_\epsilon = (-\epsilon, 
\epsilon) \times (-\epsilon, 1 + \epsilon)$ for $\epsilon > 0$.  The 
space-filling curve previously constructed as in Chapter 7 of \cite{Ste} 
has already been 
shown to be Lipschitz when defined as a function from a set which is 
biLipschitz to $S$ with image the unit square.  We can extend this 
mapping to a Lipschitz mapping $F$ from $U_\epsilon$ to $\textbf{R}^2$ by 
standard constructions (note the importance of having a Euclidean target 
space here).  

However, there does not exist a countable collection of sets $A_1, 
\dots, A_n, \dots$ such that $G := U_\epsilon \backslash \bigcup_n A_n$ is sent to 
a set of arbitrarily small Hausdorff content by $F$ and $F$ is 
biLipschitz when restricted to the $A_n$.  This is because $A_n \cap S$ 
must be a nullset 
(by the previous arguments concerning the space-filling curve for each 
$G$) which implies that $G$ must contain almost all of $S$, in the sense 
of Hausdorff measure.  Therefore, $F(G)$ must contain almost all of the 
unit square in the sense of Hausdorff measure (or Hausdorff content, 
which is equivalent in this case), producing our desired contradiction.
\end{proof}

In this example, the first and second properties (involving 
homogeneity) from Section 2.3 fail, which suggests that some notion 
of homogeneity is also necessary for the results in \cite{J} to extend to 
other spaces.

\bigskip
\bigskip
Department of Mathematics and Statistics, Helsingin yliopisto, Helsinki, Finland

{\it E-mail address}:  william.meyerson@helsinki.fi

\end{document}